\documentclass[runningheads]{llncs}
\usepackage[T1]{fontenc}
\usepackage{graphicx}
\usepackage{amsmath,amsfonts,amssymb}
\usepackage{bm}
\usepackage{tikz,tikz-3dplot}
\usetikzlibrary{shapes}
\usetikzlibrary{patterns,snakes}
\usepackage{comment}

\begin{document}
\title{On cone partitions for the min-cut and max-cut problems with non-negative edges\thanks{The research is supported by the P.G. Demidov Yaroslavl State University Project VIP-016.}}
\titlerunning{On cone partitions for the min-cut and max-cut problems}
% If the paper title is too long for the running head, you can set
% an abbreviated paper title here
%
\author{Andrei V. Nikolaev\inst{1}\orcidID{0000-0003-4705-2409} \and
Alexander V. Korostil\inst{2}\orcidID{0000-0003-1881-0207}}
\authorrunning{A.V. Nikolaev and A.V. Korostil}
% First names are abbreviated in the running head.
% If there are more than two authors, 'et al.' is used.
%
\institute{P.G. Demidov Yaroslavl State University, Yaroslavl, Russia
\email{andrei.v.nikolaev@gmail.com}\\
%\url{http://www.springer.com/gp/computer-science/lncs} 
\and
Tinkoff Bank, Moscow, Russia\\
\email{av.korostil@gmail.com}}
\maketitle              % typeset the header of the contribution
\begin{abstract}
We consider the classical minimum and maximum cut problems: find a partition of vertices of a graph into two disjoint subsets that minimize or maximize the sum of the weights of edges with endpoints in different subsets. It is known that if the edge weights are non-negative, then the min-cut problem is polynomially solvable, while the max-cut problem is NP-hard.

We construct a partition of the positive orthant into convex cones corresponding to the characteristic cut vectors, similar to a normal fan of a cut polyhedron. A graph of a cone partition is a graph whose vertices are cones, and two cones are adjacent if and only if they have a common facet. We define adjacency criteria in the graphs of cone partitions for the min-cut and max-cut problems. Based on them, we show that for both problems the vertex degrees are exponential, and the graph diameter equals 2. These results contrast with the clique numbers of graphs of cone partitions, which are linear for the minimum cut problem and exponential for the maximum cut problem.

\end{abstract}

\section{Introduction}

We consider two classical problems of finding a cut in an undirected graph.

\vspace{2mm}

\textbf{\textsc{Minimum and maximum cut problems.}}

\textsc{Instance.} Given an undirected graph $G=(V,E)$ with an edge weight function $w:E \rightarrow \mathbb{R}_{\geq 0}$.

\textsc{Question.} Find a subset of vertices $S \subset V$ such that the sum of the weights of the edges from $E$ with one endpoint in $S$ and another in $V \backslash S$ is as small as possible (\textit{minimum cut} or \textit{min-cut}) or as large as possible (\textit{maximum cut} or \textit{max-cut}).

\vspace{2mm}

Both problems have many practical applications.
The minimum cut problem is most often associated with the max-flow min-cut theorem of Ford and
Fulkerson: the maximum flow in the flow network is equal to the total weight of the edges in a minimum cut~\cite{Ford1956}.
In particular, it is used in planning communication networks and determining their reliability~\cite{Lun2008,Picard1980}.
In turn, the maximum cut problem arises in cluster analysis~\cite{Boros1989}, the Ising model in statistical physics~\cite{Barahona1988}, the VLSI design~\cite{Barahona1988,Chen1983}, and the image segmentation~\cite{deSousa2013}.

In terms of computational complexity, the min-cut problem with non-negative edges is polynomially solvable, for example, by Dinic-Edmonds-Karp flow algorithm in $O (|V|^{3} |E|)$ time~\cite{Edmonds1972}, or by Stoer-Wagner algorithm in $O(|V| |E| + |V|^2 \log |V|)$ time~\cite{Stoer1997}.
On the other hand, the min-cut and max-cut problems with arbitrary edges, and the max-cut problem with non-negative edges are known to be NP-hard~\cite{Garey1979,Karp1972}. More background information on the min-cut and max-cut problems can be found in the Encyclopedia of Optimization~\cite{Pardalos2009} and the Handbook of Combinatorial Optimization~\cite{Pardalos2013}.

In this paper, we approach to the cut problem from a polyhedral point of view by studying the properties of the graphs of cone partitions of a positive orthant with respect to the characteristic cut vectors.
The results of the research are summarized in Table~\ref{Table_pivot_1-skeleton_cut} and highlighted in bold.

\begin{table}[h]
	\centering
	\caption{Pivot table of properties of the graphs of cone partitions for the cut problems in the complete graph $K_n$}
	\label{Table_pivot_1-skeleton_cut}
	\resizebox{\textwidth}{!}{%
		\begin{tabular}{c|c|c|c}
			& Arbitrary & Minimum & Maximum \\
			& cut & non-negative & non-negative \\
			& & cut & cut \\ \hline
			Vertex adjacency & $O(1)$~\cite{Barahona1986} & $\bm{O(n)}$ & $\bm{O(n)}$ \\ \hline
			Diameter & $1$~\cite{Barahona1986} & $\bm{2}$ & $\bm{2}$ \\ \hline
			Vertex degree & $2^{n-1}-1$~\cite{Barahona1986} & $\bm{2^{n-k} + 2^{k} - 4}$ & $\bm{2^{n-1} - 2^k - 2^{n-k} + 2 + n}$ \\ \hline
			Clique number & $2^{n-1}$~\cite{Barahona1986} & ${2n - 3}$~\cite{Bondarenko2013} & ${{n \choose \frac{n}{2}-1}}$ or ${{n \choose \frac{n-1}{2}}}$~\cite{Bondarenko2016}
		\end{tabular}
	}
\end{table}

\section{Cut polytope and cone partition}

We consider a complete graph $K_n = (V,E)$ on $n$ vertices.
With each subset $S \subseteq V$ we associate the characteristic $0/1-$vector $\mathbf {v}(S) \in \{0,1\}^{d}$, where $d = C^{2}_{n}$ and
\[
\mathbf{v}(S)_{i,j} =
\begin{cases}
	1, &\text{if } |S \cap \{i,j\}| = 1,\\
	0, &\text{otherwise}.
\end{cases}
\]
Thus, the coordinates of the characteristic vector (also known as the \textit{cut vector}) indicate whether the corresponding edges are in the cut or not.

The \textit{cut polytope} $\mathrm{CUT}(n)$ (see Barahona and Mahjoub~\cite{Barahona1986}) is defined as the convex hull of all characteristic (cut) vectors:
\[\mathrm{CUT}(n) = \operatorname{conv} {\{\mathbf{v} (S): S \subseteq V\}} \subset \mathbb{R}^{d}.\]
An example of constructing the cut polytope $\mathrm{CUT}(3)$ is shown in Fig.~\ref{Fig_CUT(3)}.

\begin{figure}[t]
	\centering
	\begin{tikzpicture}[scale=0.95]
		\node[circle,draw,inner sep=3pt,fill=blue,fill opacity = 0.25] (A) at (-30:1) {};
		\node[circle,draw,inner sep=3pt,fill=blue,fill opacity = 0.25] (B) at (90:1) {};
		\node[circle,draw,inner sep=3pt,fill=blue,fill opacity = 0.25] (C) at (210:1) {};
		
		\draw (A) -- (B) -- (C) -- (A);
		
		\begin{scope}[xshift=3cm]
			\node[circle,draw,inner sep=3pt,fill=red,fill opacity = 0.25] (A) at (-30:1) {};
			\node[circle,draw,inner sep=3pt,fill=blue,fill opacity = 0.25] (B) at (90:1) {};
			\node[circle,draw,inner sep=3pt,fill=blue,fill opacity = 0.25] (C) at (210:1) {};
			
			\draw (A) -- (B) -- (C) -- (A);
			
			\draw [thick,red] (A) -- (B);
			\draw [thick,red] (A) -- (C);
			
			\draw [purple,dashed,thick] (45:1.2) -- (255:1.2);
		\end{scope}

		\begin{scope}[xshift=6cm]
			\node[circle,draw,inner sep=3pt,fill=blue,fill opacity = 0.25] (A) at (-30:1) {};
			\node[circle,draw,inner sep=3pt,fill=red,fill opacity = 0.25] (B) at (90:1) {};
			\node[circle,draw,inner sep=3pt,fill=blue,fill opacity = 0.25] (C) at (210:1) {};
			
			\draw (A) -- (B) -- (C) -- (A);
			
			\draw [thick,red] (B) -- (A);
			\draw [thick,red] (B) -- (C);
			
			\draw [purple,dashed,thick] (15:1.2) -- (165:1.2);
		\end{scope}

		\begin{scope}[xshift=9cm]
			\node[circle,draw,inner sep=3pt,fill=blue,fill opacity = 0.25] (A) at (-30:1) {};
			\node[circle,draw,inner sep=3pt,fill=blue,fill opacity = 0.25] (B) at (90:1) {};
			\node[circle,draw,inner sep=3pt,fill=red,fill opacity = 0.25] (C) at (210:1) {};
			
			\draw (A) -- (B) -- (C) -- (A);
			
			\draw [thick,red] (C) -- (A);
			\draw [thick,red] (C) -- (B);
			
			\draw [purple,dashed,thick] (135:1.2) -- (285:1.2);
		\end{scope}
		
		\node at (4.5,-1.5) {$\mathrm{CUT}(3) = \operatorname{conv}\left\lbrace (0,0,0), (0,1,1), (1,0,1), (1,1,0) \right\rbrace$};
		
	\end{tikzpicture}
	\caption {An example of constructing a cut polytope for $K_3$}
	\label {Fig_CUT(3)}
\end{figure}
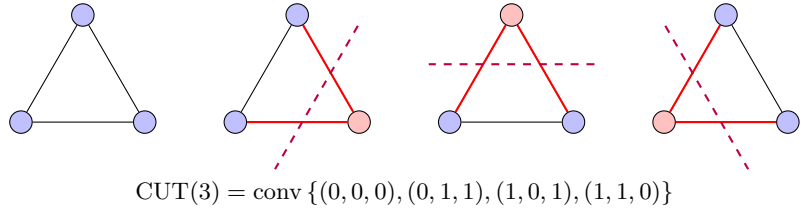

The cut polytope and its various relaxations often serve as the linear programming models for the cut problem. See, for example, the polynomial time algorithm by Barahona for the max-cut problem on graphs not contractible to $K_5$~\cite{Barahona1983}.

We consider a dual construction of a cone partition (see Bondarenko~\cite{Bondarenko1983,Bondarenko1993}), similar to a \textit{normal fan} (see, for example, Ziegler~\cite{Ziegler1995}). Let $\mathrm{X}$ be some set of points in $\mathbb{R}^{d}$ (for example, all cut vectors $\mathbf{v}(S)$ for $S \subseteq V$), and $\mathbf{x} \in \mathrm{X}$. Denote by
\[ K(\mathbf{x}) = \{\mathbf{c} \in \mathbb{R}^{d}:\ \langle \mathbf{c}, \mathbf{x} \rangle \geq \langle \mathbf{c}, \mathbf{y} \rangle,\ \forall \mathbf{y} \in \mathrm{X}\},\]
where $\langle \mathbf{c}, \mathbf{x} \rangle = \mathbf{c}^T \mathbf{x}$ is the scalar product. 
Thus, $K(\mathbf{x})$ as a set of solutions to a system of linear homogeneous inequalities is a convex polyhedral cone that includes all points $\mathbf{c} \in \mathbb{R}^d$, for which the linear function $\mathbf{c}^T \mathbf{x}$ achieves its maximum on the set $\mathrm{X}$ at the point $\mathbf{x}$.
The collection of all cones $K(\mathbf{x})$ is called the \textit{cone partition of the space $\mathbb{R}^{d}$ with respect to the set $\mathrm{X}$} (Fig.~\ref{Fig_pairwise_adjacent_cones}).
The cone partition is analogous to the Voronoi diagram, exactly coinciding with it if the Euclidean norms of all points of the set $\mathrm{X}$ are equal.

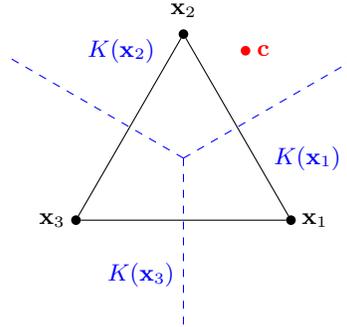
\begin{figure}[t]
	\centering	
	\begin{tikzpicture}[scale=1.65]
		\begin{scope}[every node/.style={circle,thick,draw,inner sep=1pt,fill=black,color=black}]
			\node[label=right: $\mathbf{x}_1$] (x1) at (-30:1) {};
			\node[label=above: $\mathbf{x}_2$] (x2) at (90:1) {};
			\node[label=left: $\mathbf{x}_3$] (x3) at (210:1) {};
		\end{scope}
		
		\draw (x1) -- (x2) -- (x3) -- (x1);
		
		\draw [color=blue,dashed] (0,0) -- (30:1.65);
		\draw [color=blue,dashed] (0,0) -- (150:1.65);
		\draw [color=blue,dashed] (0,0) -- (270:1.4);
		
		\node [color=blue] at (0:1) {$K(\mathbf{x}_1)$};
		\node [color=blue] at (120:1) {$K(\mathbf{x}_2)$};
		\node [color=blue] at (250:1) {$K(\mathbf{x}_3)$};
		
		\begin{scope}[every node/.style={circle,thick,draw,inner sep=1pt,fill=red,color=red}]
			\node[label=right: $\mathbf{c}$] (c) at (60:1) {};
		\end{scope}
		
%		\draw [color=red,dashed] (-0.25,1) -- (0.25,-0.25) -- node[below,sloped] {\scriptsize $\mathbf{a}^T \mathbf{x} = b$} (0.7,-1.375);
		
	\end{tikzpicture}
	\caption {An example of a 1-skeleton and cone partition}
	\label {Fig_pairwise_adjacent_cones}
\end{figure}

The 1-\textit{skeleton} of a polytope $P$ is the graph whose vertex set is the vertex set of $P$ and the edge set is the set of geometric edges or one-dimensional faces of $P$.
The cone partition of space is directly related to the 1-skeleton of a polytope since two vertices $\mathbf{x}_1$ and $\mathbf{x}_2$ of the polytope $\operatorname{conv}(\mathrm{X})$ are adjacent if and only if the cones $K(\mathbf{x_1} )$ and $K(\mathbf{x_2})$ have a common facet (see Bondarenko~\cite{Bondarenko1983}):
\[
\mathbf{x}_1 \text{ and } \mathbf{x}_2 \text{ adjacent } \Leftrightarrow \ \dim \left( K\left( \mathbf{x}_1\right)  \cap K\left( \mathbf{x}_2\right)  \right) = d - 1.
\]
We call such two cones \textit{adjacent} and consider \textit{the graph of a cone partition} of the space $\mathbb{R}^{d}$ with respect to the set $\mathrm{X}$. In the general case, it coincides with the 1-skeleton of a polytope.

The study of the 1-skeleton is of interest, since, on the one hand, the vertex adjacency can be directly applied to develop simplex-like combinatorial optimization algorithms that move from one feasible solution to another along the edges of the 1-skeleton.
See, for example, the set partitioning algorithm by Balas and Padberg \cite{Balas1975}, Balinski's algorithm for the assignment problem \cite{Balinski1985}, Ikura and Nemhauser's algorithm for the set packing \cite{Ikura1985}, etc.

On the other hand, some characteristics of the 1-skeleton estimate the time complexity for different computation models and classes of algorithms.
In particular, \textit{the diameter} (the greatest distance between any pair of vertices) is a lower bound for the number of iterations of the simplex method and similar algorithms.
Indeed, let the shortest path between a pair of vertices $\mathbf{u}$ and $\mathbf{v}$ of a polytope $P$ consist of $d(P)$ edges. If a simplex-like algorithm chooses $\mathbf{u}$ as the initial solution, and the optimal solution is $\mathbf{v}$, then no matter how successfully the algorithm chooses the next adjacent vertex of the 1-skeleton, the number of iterations cannot be less than $d(P)$ (see, for example, Dantzig~\cite{Dantzig1963}).

Note that although the diameter of a graph can easily be found in polynomial time in the number of vertices, combinatorial polytopes tend to have exponentially many vertices.
In general, it is NP-hard to determine the diameter of a 1-skeleton of a polytope specified by linear inequalities with integer data (see Frieze and Teng~\cite{Frieze1994}).

Another important characteristic is the \textit{clique number} of the 1-skeleton of the polytope $P$ (the number of vertices in the largest clique), which serves as a lower bound on the worst-case complexity in the class of algorithms based on linear decision trees. See Bondarenko~\cite{Bondarenko1993} for more details. Besides, for all known cases, it has been established that the clique number of 1-skeleton of a polytope is polynomial for polynomially solvable problems~\cite{Bondarenko2018,Maksimenko2004,Nikolaev2022} and superpolynomial for intractable problems~\cite{Bondarenko1983,Bondarenko2017,Padberg1989,Simanchev2018}.

Returning to the cut polytope, the properties of its 1-skeleton were studied by Barahona and Mahjoub in~\cite{Barahona1986}.
\begin{theorem}[Barahona and Mahjoub~\cite{Barahona1986}]
	The 1-skeleton of the $\mathrm{CUT}(n)$ polytope is a complete graph.
\end{theorem}
Thus, any two vertices of the cut polytope $\mathrm{CUT}(n)$ are adjacent, which makes the 1-skeleton not very useful in this case.

However, the polytope $\mathrm{CUT}(n)$ is associated with the cut problem in the graph with arbitrary edge weights and does not reflect the differences between max-cut and min-cut problems with non-negative edges. Since with arbitrary edges, both min-cut and max-cut problems are equivalent and NP-hard~\cite{Garey1979}.

To take into account the specifics of the cut problem, the construction of a \textit{cut polyhedron} is introduced (see Conforti et al.~\cite{Conforti2004}), which is the dominant of a cut polytope:
\[\operatorname{dmt} (\mathrm{CUT}(n)) = \mathrm{CUT}(n) + \mathbb{R}^{d}_{+},\]
i.e. the Minkowski sum of a polytope and a positive orthant.

In this paper, we consider the dual construction of a cone partition of a positive orthant with respect to the set of cut vectors for all non-empty cuts $X \subset V$ in the complete graph $K_n=(V,E)$:
\begin{align*}
	K^{+}_{\max}(X) &= \{\mathbf{c} \in \mathbb{R}^{d},\ \mathbf{c} \geq \mathbf{0}:\ \langle \mathbf{c}, \mathbf{v} (X) \rangle \geq \langle \mathbf{c}, \mathbf{v} (Y) \rangle,\ \forall Y \subset V\},\\
	K^{+}_{\min}(X) &= \{\mathbf{c} \in \mathbb{R}^{d},\ \mathbf{c} \geq \mathbf{0}:\ \langle \mathbf{c}, \mathbf{v} (X) \rangle \leq \langle \mathbf{c}, \mathbf{v} (Y) \rangle,\ \forall Y \subset V\}.
\end{align*}
Firstly, these cone partitions and their graphs were introduced in~\cite{Bondarenko2013} and later studied in~\cite{Bondarenko2016}. 
In particular, it was established that the clique number of a graph of a cone partition is linear for the polynomially solvable min-cut problem and exponential for the NP-hard max-cut problem.
Similar results for the cut polyhedron and the minimum cut problem are considered in~\cite{Conforti2004,Skutella2010}.

Note that we exclude the empty cut from consideration and identify each cut $X \subset V$ and its complement $\bar{X} = V \backslash X$. Thus, the total number of cuts is $2^{|V|-1}-1$.

\section{Vertex adjacency}

The adjacency criteria in the graphs of cone partitions for cut problems with non-negative edges were introduced in~\cite{Bondarenko2013}. In this section, we present simpler alternative versions of the criteria with new proofs based on crossing sets terminology and the submodularity of the cut function.

Two subsets $A,B \subset V$ are called \textit{crossing} if
\[A \cap B \neq \emptyset, \text{ and } A \backslash B \neq \emptyset, \text{ and } B \backslash A \neq \emptyset, \text{ and } V \backslash (A \cup B) \neq \emptyset.\]
An example of crossing sets is shown in Fig.~\ref{Fig_crossing_sets}.

\begin{figure}[t]
	\centering
	\begin{tikzpicture}[scale=0.75]

		\begin{scope}[every node/.style={circle,thick,draw,inner sep=3pt}]
			\node (W) at (-2,0) {};
			\node (N) at (0,1) {};
			\node (E) at (2,0) {};
			\node (S) at (0,-2) {};
		\end{scope}
		
		\draw [rotate around={45:(-1,-1)},blue] (-1,-1) ellipse (1 and 2.5);
		\node [blue] at (-1.2,-0.8) {$A$};
		
		\draw [red] [rotate around={-45:(1,-1)}] (1,-1) ellipse (1 and 2.5);
		\node [red] at (1.2,-0.8) {$B$};
		
	\end{tikzpicture}
	\caption{An example of crossing sets}
	\label {Fig_crossing_sets}
\end{figure}
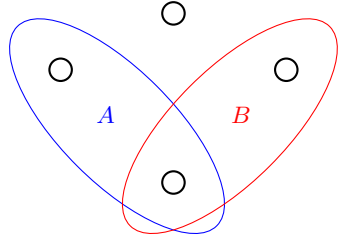

\begin{theorem} \label{Theorem_mincut_adjacency}
	The cones $K^{+}_{\min}(X)$ and $K^{+}_{\min}(Y)$ are adjacent if and only if the cuts $X$ and $Y$ are not crossing.
\end{theorem}

\begin{proof}
	Suppose that the cuts $X$ and $Y$ are crossing, but the cones $K^{+}_{\min}(X)$ and $K^{+}_{\min}(Y)$ are adjacent. The adjacency of cones means that there exists a non-negative vector $\mathbf{c}$ that belongs to both cones $K^{+}_{\min}(X)$ and $K^{+}_{\min}(Y)$ but does not belong to any other cone from the partition $K^{+}_{\min}$:
	\begin{equation}
		\exists \mathbf{c}\in \mathbb{R}^{d}\ (\mathbf{c} \geq \mathbf{0}):\ \langle \mathbf{c},\mathbf{v}(X) \rangle = \langle \mathbf{c},\mathbf{v}(Y) \rangle < \langle \mathbf{c},\mathbf{v}(Z) \rangle,\ \forall Z \subset V\  (Z \neq X,Y).
		\label{Eq_cone_facet_ineq}
	\end{equation}
	
	An important property of the cut function is \textit{submodularity} (see Schrijver~\cite{Schrijver2003}):
	\[\langle \mathbf{c},\mathbf{v}(X) \rangle + \langle \mathbf{c},\mathbf{v}(Y) \rangle \geq \langle \mathbf{c},\mathbf{v}(X \cup Y) \rangle + \langle \mathbf{c}, \mathbf{v} (X \cap Y) \rangle.\]
	Since the cuts $X$ and $Y$ are crossing, both cuts $X \cup Y$ and $X \cap Y$ exist and are not empty. Moreover, the value of at least one of them does not exceed $\langle \mathbf{c},\mathbf{v}(X) \rangle$ and $\langle \mathbf{c},\mathbf{v}(Y) \rangle $ due to the submodularity of the cut function. 
	Therefore, the inequality (\ref{Eq_cone_facet_ineq}) is violated, and the cones $K^{+}_{\min} (X)$ and $K^{+}_{\min} (Y)$ are not adjacent.
	
	Now suppose that the cuts $X$ and $Y$ are not crossing. It is easy to check that in this case at least one of the cuts or its complement is a subset of another cut or its complement. Without loss of generality, we assume that $X \subset Y$.
	
	We consider the following vector $\mathbf{c}$ of edge weights (Fig.~\ref{Fig_cut_xy_nested}):
	\begin{itemize}
		\item the total weight of the edges between $X$ and $Y \backslash X$, and between $\bar{Y}$ and $Y \backslash X$ are both equal to $2$;
		\item edges between $X$ and $\bar{Y}$ have total weight $1$;
		\item all other edges have weight $4$.
	\end{itemize}

	\begin{figure}[t]
		\centering
		\begin{tikzpicture}[scale=0.8]

			\begin{scope}[every node/.style={circle,thick,draw,minimum size=10mm}]
				\node [blue] (X) at (-2,0) {$X$};
				\node [red,dashed] (notY) at (2,0) {$\bar{Y}$};
				\node (Y-X) at (0,-2) {$Y \backslash X$};
			\end{scope}
			
			\draw [rotate around={45:(-1,-1)},red] (-1,-1) ellipse (1.25 and 3);
			\node [red] at (-1.5,-1.5) {$Y$};
			
			\draw [blue,dashed] [rotate around={-45:(1,-1)}] (1,-1) ellipse (1.25 and 3);
			\node [blue] at (1.5,-1.5) {$\bar{X}$};
			
			\draw (X) -- node[left] {$2$}  (Y-X);
			\draw (notY) -- node[right] {$2$}  (Y-X);
			\draw (X) -- node[above] {$1$}  (notY);
			
		\end{tikzpicture}
		\caption{Cuts of $X$ and $Y$ in the case of $X \subset Y$}
		\label{Fig_cut_xy_nested}
	\end{figure}
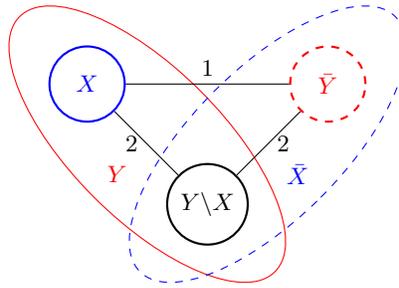
	
	By construction, the values of the cuts $X$ and $Y$ are both equal to $3$, the value of the cut $Y \backslash X$ is $4$, and all other cuts contain at least one edge of the weight $4$. Thus, by the inequality (\ref {Eq_cone_facet_ineq}), the cones $K^{+}_{\min} (X)$ and $K^{+}_{\min} (Y)$ are adjacent.
\end{proof}

\begin{theorem} \label{Theorem_maxcut_adjacency}
	The cones $K^{+}_{\max} (X)$ and $K^{+}_{\max} (Y)$ are adjacent if and only if one of the following conditions is satisfied:
	\begin{itemize}
		\item cuts $X$ and $Y$ are crossing;
		
		\item 
		the symmetric difference between cuts $X$ and $Y$ contains exactly one element
		\[|X \ominus Y| = 1, \text{ or } |\bar{X} \ominus Y| = 1, \text { or } |X \ominus \bar{Y}| = 1, \text{ or } |\bar{X} \ominus \bar{Y}| = 1.\]
	\end{itemize}
\end{theorem}

\begin{proof}
	Adjacency of the cones $K^{+}_{\max} (X)$ and $K^{+}_{\max} (Y)$ means that:
	\begin{equation}
		\exists \mathbf{c}\in \mathbb{R}^{d}\ (\mathbf{c} \geq \mathbf{0}):\ \langle \mathbf{c},\mathbf{v}(X) \rangle = \langle \mathbf{c},\mathbf{v}(Y) \rangle > \langle \mathbf{c},\mathbf{v}(Z) \rangle,\ \forall Z \subset V \ (Z \neq X,Y).
		\label{Eq_cone_facet_ineq_max}
	\end{equation}
	
	Let the cuts $X$ and $Y$ be not crossing and contain more than one element in the symmetric difference. Without loss of generality, we examine the case $X \subset Y$. Since $|X \ominus Y| > 1$, the set $Y \backslash X$ contains at least two elements and can be divided into two non-empty subsets $A$ and $B$. We consider two additional cuts $Z = X \cup A$ and $T = X \cup B$ (Fig.~\ref {Fig_cut_double_intersect}). 	By the submodularity of the cut function, we obtain that
	\[\langle \mathbf{c},\mathbf{v}(Z) \rangle + \langle \mathbf{c},\mathbf{v}(T) \rangle \geq \langle \mathbf{c},\mathbf{v}(X = Z \cap T) \rangle + \langle \mathbf{c},\mathbf{v}(Y = Z \cup T) \rangle.\]
	The value of at least one of the cuts $Z$ or $T$ cannot be less than the value of $X$ and $Y$.
	Thus, by (\ref {Eq_cone_facet_ineq_max}), the cones $K^{+}_{\max} (X)$ and $K^{+}_{\max} (Y)$ are not adjacent.
	
	\begin{figure}[t]
		\centering
		\begin{tikzpicture}[scale=0.65]
			\begin{scope}[every node/.style={circle,thick,draw,minimum size=10mm}]
				\node [blue] (x) at (0,0) {$X$};
				\node [red,dashed] (not-y) at (3,0) {$\bar{Y}$};
				\node (A) at (0,-3) {$A$};
				\node (B) at (3,-3) {$B$};
			\end{scope}
			
			\draw (not-y) -- (x);
			\draw (not-y) -- (A);
			\draw (not-y) -- (B);
			\draw (x) -- (A);
			\draw (x) -- (B);
			
			\node [red] at (1.5,-3.5) {$Y$};
			
			\draw [red] plot [smooth cycle] coordinates {(-1,1) (1,1) (1,-2) (4,-2) (4,-4) (-1,-4) };

			\begin{scope}[xshift=7cm]
				\begin{scope}[every node/.style={circle,thick,draw,minimum size=10mm}]
					\node [blue] (x) at (0,0) {$X$};
					\node [red,dashed] (not-y) at (3,0) {$\bar{Y}$};
					\node (A) at (0,-3) {$A$};
					\node (B) at (3,-3) {$B$};
				\end{scope}
				
				\draw (x) -- (not-y);
				\draw (x) -- (B);
				\draw (A) -- (not-y);
				\draw (A) -- (B);
				\draw (x) -- (A);
				\draw (not-y) -- (B);
				
				\node [purple] at (0,-4.1) {$Z$};
				\draw [purple] (0,-1.5) ellipse (1.25 and 3);

				\node [teal] at (3.8,-3.8) {$T$};
				\draw [teal] [rotate around={45:(1.5,-1.5)}] (1.5,-1.5) ellipse (1.15 and 3.75);
			\end{scope}

		\end{tikzpicture}
		\caption {Cuts $X$, $Y$, $Z$, and $T$.}
		\label {Fig_cut_double_intersect}
	\end{figure}
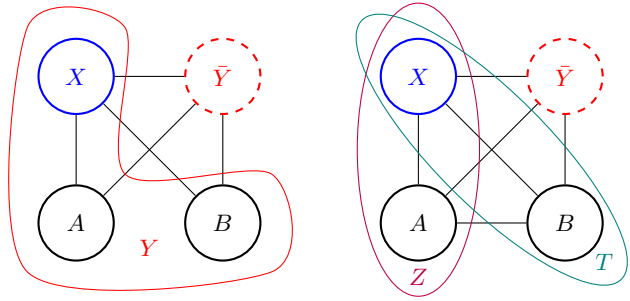

	Now suppose that the cuts $X$ and $Y$ are crossing. Again we consider the special vector $\mathbf{c}$ of edge weights (Fig.~\ref{Fig_cuts_crossing}):
	\begin{itemize}
		\item all edges between $X \cap Y$ and $\bar{X} \cap \bar{Y}$, and between $X \cap \bar {Y}$ and $\bar{X} \cap Y$ have positive weights with a total sum both equal to $1$;
		
		\item the weights of all other edges are zero.
	\end{itemize}

	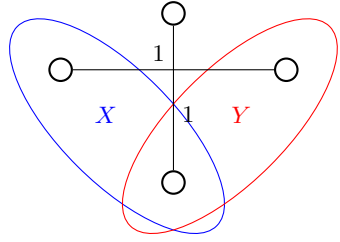
\begin{figure}[t]
		\centering
		\begin{tikzpicture}[scale=0.75]

			\begin{scope}[every node/.style={circle,thick,draw,inner sep=3pt}]
				\node (W) at (-2,0) {};
				\node (N) at (0,1) {};
				\node (E) at (2,0) {};
				\node (S) at (0,-2) {};
			\end{scope}
			
			%	\node at (0,1) {$n-k-2$};
			
			\draw [rotate around={45:(-1,-1)},blue] (-1,-1) ellipse (1 and 2.5);
			\node [blue] at (-1.2,-0.8) {$X$};
			
			\draw [red] [rotate around={-45:(1,-1)}] (1,-1) ellipse (1 and 2.5);
			\node [red] at (1.2,-0.8) {$Y$};
			
			\draw (W) -- node[above left] {$1$} (E);
			\draw (N) -- node[below right] {$1$} (S);
			
		\end{tikzpicture}
		\caption{The case of crossing cuts $X$ and $Y$}
		\label {Fig_cuts_crossing}
	\end{figure}
	
	Both cuts $X$ and $Y$ have a value of $2$ equal to the total sum of all edges in the graph. Any other cut skips at least one non-zero edge, and its value is less than $2$. Therefore, the cones $K^{+}_{\max} (X)$ and $K^{+}_{\max} (Y)$ are adjacent.
	
	It remains to consider the last configuration when the cuts $X$ and $Y$ do not cross and contain only one element in the symmetric difference. Without loss of generality, we assume $X \subset Y$ and $|X \ominus Y| = 1$.
	We consider the following vector $\mathbf{c}$ of edge weights (Fig.~\ref{Fig_cuts_single_difference}):
	\begin{itemize}
		\item each edge between $X$ and $\bar{Y}$ has a positive weight and their total sum is equal to $1$;
		\item the weights of all other edges are equal to $0$.
	\end{itemize}
	
	\begin{figure}[t]
		\centering
		\begin{tikzpicture}[scale=0.7]

			\begin{scope}[every node/.style={circle,thick,draw,inner sep=3pt}]
				\node [blue,minimum size=10mm] (W) at (-2,0) {$X$};
				\node [red,dashed,minimum size=10mm] (E) at (2,0) {$\bar{Y}$};
				\node (S) at (0,-2) {};
			\end{scope}
			
			\draw [rotate around={45:(-1,-1)},red] (-1,-1) ellipse (1.25 and 3);
			\node [red] at (-1.2,-1.2) {$Y$};
			
			\draw [blue,dashed] [rotate around={-45:(1,-1)}] (1,-1) ellipse (1.25 and 3);
			\node [blue] at (1.2,-1.2) {$\bar{X}$};
			
			\draw (W) -- node[above] {$1$} (E);
			
		\end{tikzpicture}
		\caption{Case of cuts $X$ and $Y$ with one element in the symmetric difference}
		\label {Fig_cuts_single_difference}
	\end{figure}
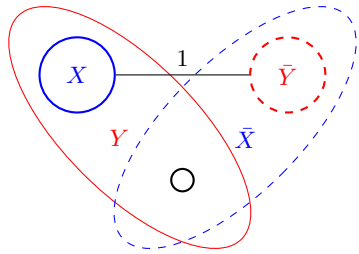
	
	Again, the total sum of the weights of all edges in the graph and the values of the cuts $X$ and $Y$ are equal to $1$. Any other cut skips at least one non-zero edge, and its value is less than $1$. 
	Thus, the cones $K^{+}_{\max} (X)$ and $K^{+}_{\max} (Y)$ are adjacent.	
\end{proof}

Note that if $X,Y \subset V$ are two cuts, then verifying whether the sets are crossing and the corresponding cones are adjacent can be done in linear time $O(V)$.

\section{Graph diameter}

In this section, we present new results on the diameter of the graphs of cone partitions for the min-cut and max-cut problems with non-negative edges.

\begin{theorem}
	The diameter $d(K^{+}_{\min})$ of the graph of cone partition for the min-cut problem with non-negative edges is equal to $2$ for all $|V| \geq 4$.
\end{theorem}	

\begin{proof}
	Cases $|V| \leq 3$ are trivial: in a graph on two vertices there is only one non-empty cut, and in a graph on three vertices, the cones of all three non-empty cuts are pairwise adjacent.
	
	Recall that \textit{the eccentricity} $\epsilon(\mathbf{v})$ of a graph vertex $\mathbf{v}$ is the greatest distance between $\mathbf{v}$ and any other vertex of a graph. Let us show that the graph of the cone partition $K^{+}_{\min}$ contains vertices with eccentricity 1, i.e. vertices that are adjacent to all others.
	We choose a cut $X$ separating exactly one element ($|X| = 1$ or $|\bar{X}| = |V|-1$). Consider an arbitrary cut $Y$ different from $X$:
	\begin{itemize}
		\item if $X \subset Y$, then $X \backslash Y = \emptyset$, the cuts are not crossing, and the corresponding cones $K^{+}_{\min}(X)$ and $K^{+ }_{\min}(Y)$ are adjacent;
		\item if $X \not\subset Y$, then $X \cap Y = \emptyset$, the cuts are not crossing, and the cones $K^{+}_{\min}(X)$ and $K^{+}_{\min}(Y)$ are also adjacent.
	\end{itemize}
	Thus, $\epsilon(K^{+}_{\min}(X)) = 1$, hence the graph diameter equals 2.
\end{proof}
	
Note that for other cuts $Y$ such that $2 \leq |Y| \leq |V|-2$, we have $\epsilon(K^{+}_{\min}(Y)) = 2$. Indeed, for each $Y$ there exists a crossing cut $Z$ such that the cones $K^{+}_{\min}(Y)$ and $K^{+}_{\min}(Z)$ are not adjacent, but there is a path between them in the graph of cone partition through $K^{+}_{\min}(X)$, where $|X| = 1$.

\begin{theorem}
	The diameter $d(K^{+}_{\max})$ of the graph of cone partition for the max-cut problem with non-negative edges is equal to 2 for all $|V| \geq 4$.
\end{theorem}	

\begin{proof}
	Again, cases $|V| \leq 3$ are trivial. Besides, if $|V| = 4$, then for any cut $X$, where $|X|=2$, the cone $K^{+}_{\max}(X)$ has eccentricity 1. Indeed, any other cut on two elements is crossing with $X$, and for any cut on 1 or 3 elements, the symmetric difference with $X$ or $V \backslash X$ contains exactly 1 element.
	
	Let us show that if $|V| \geq 5$, then the eccentricity of each vertex in the graph of a cone partition equals 2.
	We consider two arbitrary cuts $X$ and $Y$ whose cones are not adjacent and construct a cut $Z$ such that the cone $K^{+}_{\max}(Z)$ is adjacent to both $K^{+}_{\max}(X)$ and $K^{+}_{\max}(Y)$.
	
	By Theorem~\ref{Theorem_maxcut_adjacency}, the cones $K^{+}_{\max}(X)$ and $K^{+}_{\max}(Y)$ are not adjacent if and only if the cuts $X$ and $Y$ are not crossing, and the symmetric difference between cuts contains more than one element. In this case, at least one of the cuts or its complement is a subset of another cut or its complement. Without loss of generality, we assume that $X \subset Y$ and $|X \ominus Y| \geq 2$. Let us construct a cut $Z$ according to the following rules.
	
	\begin{enumerate}
		\item If $|X| > 1$, then $Z = \bar{Y} \cup \{x\}$, where $x$ is one of the elements of $X$ (Fig.~\ref{Fig_Z_barY_x}). The cuts $X$ and $Z$ are crossing, hence the cones $K^{+}_{\max}(X)$ and $K^{+}_{\max}(Z)$ are adjacent. The cut $\bar{Y}$ is a subset of $Z$ and the symmetric difference is exactly one element $x$. Therefore, the cones $K^{+}_{\max}(Y)$ and $K^{+}_{\max}(Z)$ are also adjacent. 
		
		\item If $|\bar{Y}| > 1$, then $Z = X \cup \{y\}$, where $y$ is one of the elements of $\bar{Y}$ (Fig.~\ref{Fig_Z_X_y}). Similarly to the previous case, the cuts $\bar{Y}$ and $Z$ are crossing, and $X$ is a subset of $Z$ with exactly one element in the symmetric difference. Therefore, the cones $K^{+}_{\max}(X)$ and $K^{+}_{\max}(Y)$ are adjacent to the cone $K^{+}_{\max}(Z)$.
		
		\item If $|X| = |\bar{Y}| = 1$, then $Z = X \cup \bar{Y}$ (Fig.~\ref{Fig_Z_X_barY}). The cuts $X$ and $\bar{Y}$ are subsets of $Z$ and differ from $Z$ by exactly one element. Therefore, the cones $K^{+}_{\max}(X)$ and $K^{+}_{\max}(Y)$ are adjacent to the cone $K^{+}_{\max}(Z)$.
	\end{enumerate}

	\begin{figure}[p]
		\centering
		\begin{tikzpicture}[scale=0.55]
			\begin{scope}[every node/.style={circle,thick,draw,inner sep=5pt}]
				\node at (0:3) {};
				\node at (60:3) {};
				\node at (120:3) {};
				\node at (180:3) {};
				\node at (240:3) {};
				\node at (300:3) {};
			\end{scope}

			\draw [blue] [rotate around={60:(150:2.5)}] (150:2.5) ellipse (2.5 and 1.25);
			
			\draw [red,dashed] (-90:2.5) ellipse (2.5 and 1.25);
			
			\node [blue] at (150:2.5) {$X$};
			\node [red] at (-90:2.5) {$\bar{Y}$};
			
			\draw [purple] plot [smooth cycle] coordinates {(165:4.25) (240:4.5) (325:4.75)};
			
			\node [purple] at (210:2) {$Z$};
			
		\end{tikzpicture}
		\caption{Case $|X| > 1$ and $Z = \bar{Y} \cup \{x\}$, where $x$ is one of the elements of $X$}
		\label{Fig_Z_barY_x}
	\end{figure}
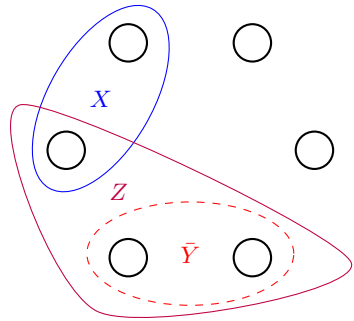

	\begin{figure}[p]
		\centering
		\begin{tikzpicture}[scale=0.55]
			\begin{scope}[every node/.style={circle,thick,draw,inner sep=5pt}]
				\node at (0:3) {};
				\node at (60:3) {};
				\node at (120:3) {};
				\node at (180:3) {};
				\node at (240:3) {};
				\node at (300:3) {};
			\end{scope}

			\draw [blue] [rotate around={60:(150:2.5)}] (150:2.5) ellipse (2.5 and 1.25);
			\draw [red,dashed] (-90:2.5) ellipse (2.5 and 1.25);
			
			\node [blue] at (150:2.5) {$X$};
			\node [red] at (-90:2.5) {$\bar{Y}$};
			
			\draw [purple] plot [smooth cycle] coordinates {(95:4.25) (150:4.25) (180:4.5) (255:4.25)};
			
			\node [purple] at (210:2) {$Z$};
			
		\end{tikzpicture}
		\caption{Case $|\bar{Y}| > 1$, then $Z = X\cup \{y\}$, where $y$ is one of the elements of $\bar{Y}$}
		\label{Fig_Z_X_y}
	\end{figure}
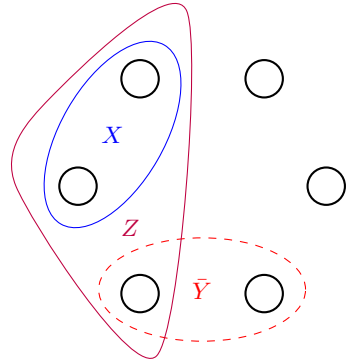

	\begin{figure}[p]
		\centering
		\begin{tikzpicture}[scale=0.55 ]
			\begin{scope}[every node/.style={circle,thick,draw,inner sep=5pt}]
				\node at (0:3) {};
				\node at (60:3) {};
				\node at (120:3) {};
				\node at (180:3) {};
				\node at (240:3) {};
				\node at (300:3) {};
			\end{scope}

			\draw [blue] (60:3) circle (1.15);
			\draw [red,dashed] (240:3) circle (1.15);
			
			\node [blue] at (60:2.2) {$X$};
			\node [red] at (240:2.2) {$\bar{Y}$};
			
			\draw [purple] [rotate around={60:(0,0)}] (0,0) ellipse (4.5 and 1.85);
			
			%		\draw [dashed,purple] plot [smooth cycle] coordinates {(70:4.25) (162:4.25) (249:4)};
			
			\node [purple] at (0,0) {$Z$};
			
		\end{tikzpicture}
		\caption{Case $|X| = |\bar{Y}| = 1$ and $Z = X \cup \bar{Y}$}
		\label{Fig_Z_X_barY}
	\end{figure}
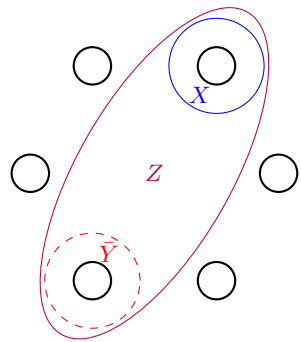
	
	Therefore, for any cuts $X$ and $Y$ there exists a cut $Z$ whose cone $K^{+}_{\max}(Z)$ is adjacent both to $K^{+}_{\max} (X)$ and $K^{+}_{\max}(Y)$. 
	On the other hand, if $|V| \geq 5$, then for any cut $X$ there is a cut $Y$ obtained from $X$ by adding or removing 2 elements, such that the cones $K^{+}_{\max} (X)$ and $ K^{+}_{\max}(Y)$ are not adjacent. Thus, the eccentricity of each vertex of the graph of a cone partition equals 2, whence the diameter of the graph is also 2.
\end{proof}

\section{Vertex degrees}

Now we study the degrees of vertices in the graphs of cone partitions for the min-cut and max-cut problems with non-negative edges.
They are of interest if the adjacency criteria are applied as neighborhood structures for simplex-like algorithms.
Let's call the \textit{cardinality of the cut} $S$ the minimum of the cardinalities $|S|$ and $|V \backslash S|$.

\begin{theorem}\label{Theorem_mincut_vertex_degree}
	Let $|V| = n$ and the cardinality of the cut $X \subset V$ equals $k$, then the degree of the vertex $K^{+}_{\min}(X)$ in the graph of the cone partition for the min-cut problem with non-negative edges is equal to
	\[2^{n-k} + 2^{k} - 4.\]
\end{theorem}	

\begin{proof}
	We separately consider two cases by the cardinality of the cut. If $k = 1$, then, by Theorem~\ref{Theorem_mincut_adjacency}, the corresponding cone is adjacent to the cones of all other cuts in the graph. Therefore, the degree of a vertex in the graph of cone partition is equal to $2^{n-1} - 2$.
	
	Now we examining the case $k > 1$. Consider some cut $X$ and its complement $\bar{X} = V \backslash X$ (see Fig.~\ref{Fig_X_cross_Y}). By definition, a cut $Y$ crossing with $X$ must contain some elements from $X$ and some elements from $\bar{X}$. The number of subsets of a finite set of $k$ elements, excluding the empty set and the entire set, is $2^{k}-2$. Consider all combinations of admissible subsets in $X$ and $\bar{X}$, divided by 2 to exclude complements, and we get that the total number of cuts that cross $X$ is
	\[ \frac{(2^k-2)\cdot(2^{n-k}-2)}{2} = 2^{n-1} - 2^k - 2^{n-k} + 2. \] 
	
	The degree of the vertex in the graph of the cone partition corresponding to the cut $X$ can be obtained by subtracting the number of crossing cuts from the total number of cuts, different from $X$:
	\begin{align*}
		\operatorname{deg} (K^{+}_{\min}(X)) &= 2^{n-1} - 2 - \left( 2^{n-1} - 2^k - 2^{n-k} + 2 \right) \\
		&= 2^{n-k} + 2^{k} - 4. 
	\end{align*}

	\begin{figure}[t]
		\centering
		\begin{tikzpicture}[scale=0.85]
			\begin{scope}[every node/.style={circle,thick,draw,inner sep=5pt}]
				\node at (0,0) {};
				\node at (1,0) {};
				\node at (2,0) {};
				\node at (3,0) {};
				\node at (4,0) {};
				\node at (5,0) {};
				\node at (6,0) {};
				\node at (7,0) {};
				\node at (8,0) {};
			\end{scope}

			\draw [blue] (1.5,0) ellipse (2 and 1);
			\draw [blue,dashed] (6,0) ellipse (2.5 and 1);

			\draw [red] (4,0) ellipse (2.5 and 1.25);
			
			\node [blue] at (1,-0.5) {$X$};
			\node [blue] at (7,-0.5) {$\bar{X}$};
			\node [red] at (3.5,-0.75) {$Y$};
			
			\draw [
			thick,
			decoration={
				brace,
				mirror,
				raise=0.5cm
			},
			decorate
			] (-0.25,-1) -- (3.25,-1) 
			node [pos=0.5,anchor=north,yshift=-0.55cm] {$k$}; 
			
			\draw [
			thick,
			decoration={
				brace,
				mirror,
				raise=0.5cm
			},
			decorate
			] (3.75,-1) -- (8.25,-1) 
			node [pos=0.5,anchor=north,yshift=-0.55cm] {$n-k$}; 
			
		\end{tikzpicture}
		\caption{The cut $X$, its complement $\bar{X}$, and the crossing cut $Y$}
		\label{Fig_X_cross_Y}
	\end{figure}
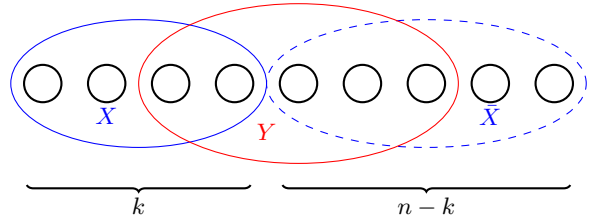
\end{proof}

\begin{corollary}
	The degree of a vertex in the graph of the cone partition for the min-cut problem with non-negative edges
	is bounded above and below by
	\[
	2^{\lceil \frac{n}{2} \rceil} + 2^{\lfloor \frac{n}{2} \rfloor} - 4 \leq \operatorname{deg} (K^{+}_{\min}(X)) \leq 2^{n-1} - 2.
	\]	
\end{corollary}

We now turn to the max-cut problem.

\begin{theorem} \label{Theorem_maxcut_vertex_degree}
	Let $|V| = n$ and the cardinality of the cut $X \subset V$ equals $k$, then the degree of the vertex $K^{+}_{\max}(X)$ in the graph of the cone partition for the max-cut problem with non-negative edges is equal to
	\[
	\begin{cases}
		n-1, &\text{if } $k = 1$,\\
		2^{n-1} - 2^{k} - 2^{n-k} + 2 + n, &\text{otherwise}.
	\end{cases}
	\]
\end{theorem}	

\begin{proof}
	Similarly, we consider separately the cases $k = 1$ and $k > 1$.
	Let $k = 1$. A unit cut $X$ cannot cross with any other cut.
	Therefore, the cone $K^{+}_{\max}(X)$ is adjacent only to the cones $K^{+}_{\max}(Y)$ for which $|X \ominus Y| = $1, i.e. $X \subset Y$ and $|Y \backslash X| = $1. There are exactly $n-1$ such cuts in total.
	
		\begin{figure}[t]
		\centering
		\begin{tikzpicture}[scale=0.9]
			\begin{scope}[every node/.style={circle,thick,draw,inner sep=5pt}]
				\node at (0,0) {};
				\node at (1,0) {};
				\node at (2,0) {};
				\node at (3,0) {};
				\node at (4,0) {};
				\node at (5,0) {};
				\node at (6,0) {};
				\node at (7,0) {};
				\node at (8,0) {};
			\end{scope}

			\draw [blue] (1.5,0) ellipse (2 and 1);
			\draw [blue,dashed] (6,0) ellipse (2.5 and 1);

			\draw [red] (2,0) ellipse (1.4 and 0.75);
			\node [red] at (2,-0.5) {$Y_1$};
			
			\draw [red] (2,0) ellipse (2.6 and 1.25);
			
			\node [blue] at (0.5,-0.5) {$X$};
			\node [blue] at (6,-0.5) {$\bar{X}$};
			\node [red] at (3.5,-0.75) {$Y_2$};
			
			\draw [
			thick,
			decoration={
				brace,
				mirror,
				raise=0.5cm
			},
			decorate
			] (-0.25,-1) -- (3.25,-1) 
			node [pos=0.5,anchor=north,yshift=-0.55cm] {$k$}; 
			
			\draw [
			thick,
			decoration={
				brace,
				mirror,
				raise=0.5cm
			},
			decorate
			] (3.75,-1) -- (8.25,-1) 
			node [pos=0.5,anchor=north,yshift=-0.55cm] {$n-k$}; 
			
		\end{tikzpicture}
		\caption{An example of cuts $X,Y_1,Y_2$, such that $|X \ominus Y_1| = |X \ominus Y_2| = 1$}
		\label{Fig_X_Y_1_element}
	\end{figure}
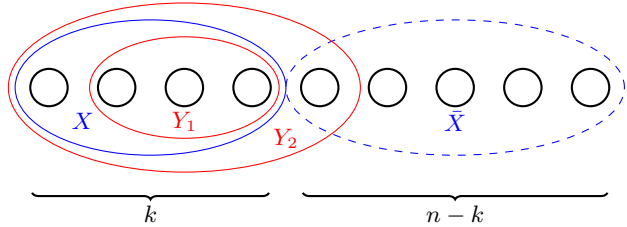
	
	Now consider some cut $X$ of cardinality $k > 1$.
	As previously stated, the number of cuts that are crossing with $X$ is equal to 
	\[2^{n-1} - 2^{k} - 2^{n-k} + 2.\]
	The number of cuts $Y$ for which $|X \ominus Y| = 1$ is equal to $n$ since there are $k$ ways to subtract one element from the cut $X$ and $n-k$ ways to add one element to $X$ (see Fig.~\ref{Fig_X_Y_1_element}). Therefore, the degree of the vertex corresponding to the cut $X$ is equal to
	\[\operatorname{deg}(K^{+}_{\max}(X)) = 2^{n-1} - 2^{k} - 2^{n-k} + 2 + n.\]
\end{proof}

\begin{corollary}
	The degree of a vertex in the graph of the cone partition for the max-cut problem with non-negative edges is bounded above and below by
	\[n-1 \leq \operatorname{deg} (K^{+}_{\max}(X)) \leq 2^{n-1} - 2^{\lceil \frac{n}{2} \rceil} - 2^{\lfloor \frac{n}{2} \rfloor} + 2 + n.\]
\end{corollary}

\section{Conclusion}

The results of the research are summarized in the pivot Table~\ref{Table_pivot_1-skeleton_cut}.
The cut problem with arbitrary edges and the max-cut problem with non-negative edges are NP-hard. On the other hand, the min-cut problem with non-negative edges is polynomially solvable.
Studying the 1-skeleton and the graphs of cone partitions associated with the cut problem, we see that in all cases verifying the adjacency is a simple problem, the diameter of the polyhedral graph does not exceed 2, and the degrees of vertices are exponential. Of all the characteristics of a 1-skeleton, only the clique number correlates with the complexity of the problem: linear for the polynomially solvable min-cut problem with non-negative edges, and superpolynomial for the NP-hard max-cut problem with non-negative edges and cut problems with arbitrary edges. 
Besides, the adjacency criteria and the structure of the graphs of cone partitions can be applied to develop and analyze simplex-like combinatorial algorithms for cut problems with non-negative edges.

\subsubsection*{Acknowledgements.}
We are very grateful to the anonymous reviewers for their comments and suggestions which helped to improve the
presentation of the results in this paper.

%\subsubsection{Acknowledgements} Please place your acknowledgments at
%the end of the paper, preceded by an unnumbered run-in heading (i.e.
%3rd-level heading).

%
% ---- Bibliography ----
%
% BibTeX users should specify bibliography style 'splncs04'.
% References will then be sorted and formatted in the correct style.
%
 \bibliographystyle{splncs04}
 \bibliography{Nikolaev_Korostil_MOTOR2023}

\begin{thebibliography}{10}
\providecommand{\url}[1]{\texttt{#1}}
\providecommand{\urlprefix}{URL }
\providecommand{\doi}[1]{https://doi.org/#1}

\bibitem{Balas1975}
Balas, E., Padberg, M.: On the set-covering problem: Ii. an algorithm for set
  partitioning. Operations Research  \textbf{23}(1),  74--90 (1975).
  \doi{10.1287/opre.23.1.74}

\bibitem{Balinski1985}
Balinski, M.L.: Signature methods for the assignment problem. Operations
  Research  \textbf{33}(3),  527--536 (1985). \doi{10.1287/opre.33.3.527}

\bibitem{Barahona1988}
Barahona, F., Grötschel, M., Jünger, M., Reinelt, M.: An application of
  combinatorial optimization to statistical physics and circuit layout design.
  Operations Research  \textbf{36}(3),  493--513 (1988).
  \doi{10.1287/opre.36.3.493}

\bibitem{Barahona1986}
Barahona, F., Mahjoub, A.R.: On the cut polytope. Mathematical Programming
  \textbf{36}(2),  157--173 (1986). \doi{10.1007/BF02592023}

\bibitem{Barahona1983}
Barahona, F.: {The max-cut problem on graphs not contractible to $K_5$}.
  Operations Research Letters  \textbf{2},  107--111 (1983).
  \doi{10.1016/0167-6377(83)90016-0}

\bibitem{Bondarenko2016}
Bondarenko, V., Nikolaev, A.: On graphs of the cone decompositions for the
  min-cut and max-cut problems. International Journal of Mathematics and
  Mathematical Sciences  \textbf{2016},  7863650 (2016).
  \doi{10.1155/2016/7863650}

\bibitem{Bondarenko1983}
Bondarenko, V.A.: Nonpolynomial lower bounds for the complexity of the
  traveling salesman problem in a class of algorithms. Automation and Remote
  Control  \textbf{44}(9),  1137--1142 (1983)

\bibitem{Bondarenko1993}
Bondarenko, V.A.: Estimating the complexity of problems on combinatorial
  optimization in one class of algorithms. Phys.-Dokl.  \textbf{38}(1), ~6--7
  (1993)

\bibitem{Bondarenko2013}
Bondarenko, V.A., Nikolaev, A.V.: Combinatorial and geometric properties of the
  max-cut and min-cut problems. Doklady Mathematics  \textbf{88}(2),  516--517
  (2013). \doi{10.1134/S1064562413050062}

\bibitem{Bondarenko2018}
Bondarenko, V.A., Nikolaev, A.V.: On the skeleton of the polytope of pyramidal
  tours. Journal of Applied and Industrial Mathematics  \textbf{12}(1),  9--18
  (2018). \doi{10.1134/S1990478918010027}

\bibitem{Bondarenko2017}
Bondarenko, V.A., Nikolaev, A.V., Shovgenov, D.A.: Polyhedral characteristics
  of balanced and unbalanced bipartite subgraph problems. Automatic Control and
  Computer Sciences  \textbf{51}(7),  576--585 (2017).
  \doi{10.3103/S0146411617070276}

\bibitem{Boros1989}
Boros, E., Hammer, P.L.: On clustering problems with connected optima in
  euclidean spaces. Discrete Mathematics  \textbf{75}(1),  81--88 (1989).
  \doi{10.1016/0012-365X(89)90080-0}

\bibitem{Chen1983}
Chen, R.W., Kajitani, Y., Chan, S.P.: A graph-theoretic via minimization
  algorithm for two-layer printed circuit boards. IEEE Transactions on Circuits
  and Systems  \textbf{30}(5),  284--299 (1983). \doi{10.1109/TCS.1983.1085357}

\bibitem{Conforti2004}
Conforti, M., Rinaldi, G., Wolsey, L.: On the cut polyhedron. Discrete Math.
  \textbf{277}(1),  279--285 (2004). \doi{10.1016/j.disc.2002.12.001}

\bibitem{Dantzig1963}
Dantzig, G.B.: Linear Programming and Extensions. RAND Corporation, Santa
  Monica, CA (1963). \doi{10.7249/R366}

\bibitem{Edmonds1972}
Edmonds, J., Karp, R.M.: Theoretical improvements in algorithmic efficiency for
  network flow problems. J. ACM  \textbf{19}(2),  248--264 (1972).
  \doi{10.1145/321694.321699}

\bibitem{Pardalos2009}
Floudas, C.A., Pardalos, P.M.: Encyclopedia of Optimization. Springer New York,
  NY (2009). \doi{10.1007/978-0-387-74759-0}

\bibitem{Ford1956}
Ford, L.R., Fulkerson, D.R.: Maximal flow through a network. Canadian Journal
  of Mathematics  \textbf{8},  399--404 (1956). \doi{10.4153/CJM-1956-045-5}

\bibitem{Frieze1994}
Frieze, A.M., Teng, S.H.: On the complexity of computing the diameter of a
  polytope. Computational complexity  \textbf{4}(3),  207--219 (1994).
  \doi{10.1007/BF01206636}

\bibitem{Garey1979}
Garey, M.R., Johnson, D.S.: Computers and Intractability: A Guide to the Theory
  of NP-Completeness (Series of Books in the Mathematical Sciences). W. H.
  Freeman (1979)

\bibitem{Ikura1985}
Ikura, Y., Nemhauser, G.L.: Simplex pivots on the set packing polytope.
  Mathematical Programming  \textbf{33},  123--138 (1985).
  \doi{10.1007/BF01582240}

\bibitem{Karp1972}
Karp, R.M.: Reducibility among combinatorial problems. In: Miller, R.E.,
  Thatcher, J.W., Bohlinger, J.D. (eds.) Complexity of Computer Computations.
  The IBM Research Symposia Series. pp. 85--103. Springer US, Boston, MA
  (1972). \doi{10.1007/978-1-4684-2001-2_9}

\bibitem{Lun2008}
Lun, D.S., M\'{e}dard, M., Koetter, R., Effros, M.: On coding for reliable
  communication over packet networks. Physical Communication  \textbf{1}(1),
  3--20 (2008). \doi{10.1016/j.phycom.2008.01.006}

\bibitem{Maksimenko2004}
Maksimenko, A.: Combinatorial properties of the polyhedron associated with the
  shortest path problem. Computational Mathematics and Mathematical Physics
  \textbf{44},  1611--1614 (2004)

\bibitem{Nikolaev2022}
Nikolaev, A.V.: On 1-skeleton of the polytope of pyramidal tours with
  step-backs. Siberian Electronic Mathematical Reports  \textbf{19},  674--687
  (2022). \doi{10.33048/semi.2022.19.056}

\bibitem{Padberg1989}
Padberg, M.: The boolean quadric polytope: Some characteristics, facets and
  relatives. Mathematical Programming  \textbf{45}(1),  139--172 (1989).
  \doi{10.1007/BF01589101}

\bibitem{Pardalos2013}
Pardalos, P.M., Du, D.Z., Graham, R.L.: Handbook of Combinatorial Optimization.
  Springer New York, NY (2013). \doi{10.1007/978-1-4419-7997-1}

\bibitem{Picard1980}
Picard, J.C., Queyranne, M.: On the structure of all minimum cuts in a network
  and applications. In: Rayward-Smith, V.J. (ed.) Combinatorial Optimization
  II. pp. 8--16. Springer Berlin Heidelberg, Berlin, Heidelberg (1980).
  \doi{10.1007/BFb0120902}

\bibitem{Schrijver2003}
Schrijver, A.: Combinatorial Optimization: Polyhedra and Efficiency. Springer
  Berlin Heidelberg (2003)

\bibitem{Simanchev2018}
Simanchev, R.Y.: On the vertex adjacency in a polytope of connected k-factors.
  Trudy Inst. Mat. i Mekh. UrO RAN  \textbf{24},  235--242 (2018).
  \doi{10.21538/0134-4889-2018-24-2-235-242}

\bibitem{Skutella2010}
Skutella, M., Weber, A.: On the dominant of the $s$-$t$-cut polytope: Vertices,
  facets, and adjacency. Mathematical Programming  \textbf{124}(1),  441--454
  (2010). \doi{10.1007/s10107-010-0373-7}

\bibitem{deSousa2013}
de~Sousa, S., Haxhimusa, Y., Kropatsch, W.G.: Estimation of distribution
  algorithm for the max-cut problem. In: Kropatsch, W.G., Artner, N.M.,
  Haxhimusa, Y., Jiang, X. (eds.) Graph-Based Representations in Pattern
  Recognition. pp. 244--253. Springer Berlin Heidelberg (2013).
  \doi{10.1007/978-3-642-38221-5_26}

\bibitem{Stoer1997}
Stoer, M., Wagner, F.: A simple min-cut algorithm. J. ACM  \textbf{44}(4),
  585--591 (1997). \doi{10.1145/263867.263872}

\bibitem{Ziegler1995}
Ziegler, G.: Lectures on Polytopes. Graduate Texts in Mathematics, Springer New
  York, NY (1995). \doi{10.1007/978-1-4613-8431-1}

\end{thebibliography}

\end{document}